\documentclass[11pt,reqno]{amsart}

\numberwithin{equation}{section}

\usepackage{amssymb}
\usepackage{amsmath}
\usepackage{amsbsy}
\usepackage{amscd}
\usepackage{amsthm}
\usepackage{amsfonts}
\usepackage{ifpdf}
\ifpdf \usepackage[pdftex,pdfstartview=FitH,pdfpagemode=none,colorlinks,bookmarks,linkcolor=blue]{hyperref} \else  \usepackage[hypertex]{hyperref} \fi
\usepackage[nobysame,abbrev,alphabetic]{amsrefs}
\usepackage{color}
\usepackage{enumerate}

\newtheorem{theorem}{Theorem}[section]
\newtheorem{lemma}[theorem]{Lemma}
\newtheorem{corollary}[theorem]{Corollary}

\newtheorem{definition}[theorem]{Definition}

\newtheorem{proposition}[theorem]{Proposition}

\def\dstyle { \displaystyle }

\def\p{\phi}
\def\mR {\mathbb {R}^n}

\begin{document}
\title[Optimal Transport method for higher order isoperimetric inequalities]{Some higher order isoperimetric inequalities via the method of optimal transport}

\date{December 22, 2012}
\author{Sun-Yung A. Chang \and Yi Wang
}
\address{Sun-Yung A. Chang, Department of Mathematics, Princeton University, Princeton NJ 08544, USA}
\address{ email: chang@math.princeton.edu}

\address{Yi Wang, Department of Mathematics, Stanford University, Stanford CA 94305, USA}
\address{ email: wangyi@math.stanford.edu}
\thanks{The research of the first author is partially supported
 by NSF grant DMS-0758601; the research of the second author is partially supported
 by NSF grant DMS-1205350}
\begin{abstract}
In this paper, we establish some sharp inequalities
between the volume and the integral of the $k$-th mean curvature for
$k+1$-convex domains in the Euclidean space. The results
generalize the classical Alexandrov-Fenchel
inequalities for convex domains. Our proof utilizes the method of optimal transportation.
\end{abstract}

\maketitle

\section{introduction}
Classical isoperimetric inequality for domains in Euclidean space was rigorously
established by H. A. Schwartz
\cite{Schwartz}(1884) using the method of symmetrization. Later De Giorgi
\cite{De Giorgi} gave a simpler proof based on an early argument of
Steiner. In \cite{Gromov}, Gromov established the inequality
by constructing a map from the domain to the unit ball and applying the
divergence theorem. Inspired by Gromov's idea, it became well-known that
optimal transport method is effective in establishing various sharp
geometric inequalities. See for example \cite{McCann}, \cite{Villani}, \cite{Figalli} etc..

In this paper, we apply optimal transport method to establish some sharp higher order isoperimetric inequalities
of Alexandrov-Flenchel type for a class of domains which includes the convex domains.

Suppose $\Omega\subset \mR$ is a bounded convex set. Consider the set
$$\Omega+tB:=\{x+ty|x\in \Omega, y\in B\}$$ for $t>0$.
By a theorem of Minkowski \cite{Mk}, the volume of the set is a polynomial of degree $n$,
whose expansion is given by
$$\mbox{Vol}(\Omega+tB)=\sum_{k=0}^{n}\binom {n}{k} W_{k}(\Omega)t^{k}.$$
Here $W_{k}(\Omega)$, $k=0,..., n$ are coefficients determined by the set $\Omega$, and $\binom {n}{k}=\frac{n!}{k!(n-k)!}$.
The $k$-th quermassintegral $V_{k}$ is defined as a multiple of the coefficient $W_{n-k}(\Omega)$.
\begin{equation}
V_{k}(\Omega):= \frac{\omega_{k}}{\omega_{n}}W_{n-k}(\Omega).
\end{equation}
Here $\omega_{k}$ denotes the volume of the unit $k$-ball; for an arbitrary domain
 $\Omega$, $V_{n}(\Omega)=vol(\Omega)$ denotes the volume of $\Omega$.

If the boundary $\partial \Omega$ is smooth,
the quermassintegrals can also be represented as the integrals of invariants of the second fundamental form:
Let $L_{\alpha\beta}$ be the second fundamental form on $ \partial \Omega$, and let
$\sigma_{l}(L)$ with $l=0,..., n-1$ be the $l$-th elementary symmetric function of the eigenvalues
of $L$. (Define $\sigma_{0}(L)=1$.) Then
\begin{equation}\label{def}
V_{n-k}(\Omega):=\frac{(n-k)!(k-1)!}{n!}\frac{\omega_{n-k}}{\omega_{n}}\displaystyle \int_{\partial \Omega} \sigma_{k-1}(L)d\mu,
\end{equation}
where $d\mu$ is the surface area of $\partial \Omega$. From the above definition, one can see that $V_{0}(\Omega)=1$,
and $V_{n-1}(\Omega)=\frac{\omega_{n-1}}{n\omega_n}Area(\partial \Omega)$.
As a consequence of the Alexandrov-Fenchel inequalities \cite{Alex1}, one obtains the following family of inequalities: if $\Omega$ is a convex domain
in $\mathbb{R}^{n}$ with smooth boundary, then for $0\leq l\leq n-1$,
\begin{equation}\label{af'}
\left(\frac{V_{l+1}(\Omega)}{V_{l+1}(B)}\right)^\frac{1}{l+1}\leq \left (\frac{V_{l}(\Omega)}{V_l(B)}\right)^\frac{1}{l}.
\end{equation}
For $0\leq l\leq n-2$, (\ref{af'}) is equivalent to, due to the identity (\ref{def}),
\begin{equation}\label{af}
\displaystyle\left( \int_{\partial \Omega} \sigma_{k-1}(L)d\mu\right)^{\frac{1}{n-k}}\leq \bar C\left( \int_{\partial \Omega}
\sigma_{k}(L)d\mu\right)^{\frac{1}{n-1-k}},\\
\end{equation}
where $k=n-1-l$. Here $\bar C=C(k,n)$ denotes the (sharp) constant which is obtained only when $\Omega$ is
a ball in $\mR$.
When $l=n-1$, (\ref{af'}) is the well-known isoperimetric inequality
\begin{equation}vol(\Omega)^{\frac{n-1}{n}}\leq \frac{1}{n\omega_{n}^{\frac{1}{n}}} Area(\partial\Omega). \end{equation}

An open question in the field is if (\ref{af}) holds when the domain is only $k$-convex in the sense
that $\sigma_{l} (L)(x)>0$ for all $ l \leq k$ and all $x\in \partial \Omega$. (In the following, we denote the condition of $k$-convexity by
$L\in \Gamma_k^+$.)
This is indeed the case under the additional assumption that $\Omega $ is star-shaped, as established by Guan-Li \cite{GuanLi}.
In the work of Huisken \cite{Huisken}, he has proved that (\ref{af})
holds for $k=1$ when one assumes in addition that the domain is
outward minimizing. Inequalities of the type (\ref{af}) were
also discussed in Trudinger \cite{Trudinger}. In \cite{Castillon}, Castillon has applied the method of optimal transport to
give a new proof of the Michael-Simon inequality, which in particular implies (\ref{af}) for $k=1$ with
some constant $C(n)$.  In \cite{ChangWang, ChangWang1}, we have established inequalities of the type (\ref{af}) for all $l \leq k$
with some (non-sharp) constant $C(k,n)$ when $\Omega$ is $(k+1)$-convex.

In this note, we apply the method of optimal transport to establish the following
``end version'' of the sharp inequalities in (\ref{af}).


\begin{theorem}\label{thm-1'}Let $\Omega$ be a domain in $\mR$ with smooth boundary. Suppose $ \Omega$ is 2-convex, i.e. $L\in \Gamma^+_2$. Then
\begin{equation}\label{af1}
 vol(\Omega)^{\frac{n-2}{n}}  \leq (\frac{1 }{\omega_{n}})^{\frac{2}{n}} \frac{1}{n(n-1)} \int_{\partial \Omega}Hd\mu,
\end{equation}
where $H=\sigma_1(L)$ is the mean curvature of $\partial \Omega$ and $d\mu$ is the surface area of $\partial \Omega$. The constant in the inequality is sharp and equality
holds only when $\Omega$ is a ball in $\mR$.
\end{theorem}
We also prove the inequality between the volume of $\Omega$ and the integral of $\sigma_2(L)$ with the sharp constant.
\begin{theorem}\label{thm-2'}Let $\Omega$ be a domain in $\mR$ with smooth boundary. Suppose $ \Omega$ is 3-convex, i.e. $L\in \Gamma^+_3$. Then
\begin{equation}\label{af2}
 vol(\Omega)^{\frac{n-3}{n}}  \leq (\frac{1 }{\omega_{n}})^{\frac{3}{n}} \frac{1}{3 \binom {n}{3}} \int_{\partial \Omega}\sigma_2(L)d\mu.
\end{equation}
The constant in the inequality is sharp and equality holds
only when $\Omega$ is a ball in $\mR$.
\end{theorem}


We remark that our proof indicates that the sharp version of these inequalities for $(k+1)$-convex domains
are likely to be derived by optimal transport method for all $k$; but the proof would be algebraically challenging. It also remains an open question if the additional one level of convexity
assumption (that is we assume $\Omega$ is $(k+1)$-convex instead of $k$-convex) in our proof is necessary--but the proof
we present here heavily depends on this assumption.

The main difference of Theorem \ref{thm-1'}, and \ref{thm-2'} from our previous work in \cite{ChangWang1} is that we have applied the optimal transport
map from a measure on $\Omega$ to the unit $n$-ball, instead of applying the optimal transport map from the projections of the measure on $\partial \Omega$
to the unit $(n-1)$-balls on hyperplanes of $\mR$.
Also, in order to obtain the sharp constants, we study the delicate
interplays between terms involving the tangential and normal directions of the optimal transport map. We apply the Taylor expansion and derive recursive inequalities in order to estimate each individual term in the expansion.

Finally we would like to clarify that although an optimal transport map $\bar{\nabla}\p$ from $\Omega$ to the unit ball in $\mR$
is smooth up to the boundary if $\Omega$ is strictly convex by the results of Caffarelli (\cite {Caffarelli, Caffarelli2}),
the proof of our main theorems still work if we only assume $\Omega$ is 2 or 3-convex
respectively by an approximation argument. This type of approximation process
is quite standard and has been explained in detail for example in
\cite{ChangWang1} and in many other articles in this research field.
We will not emphasize this point later in our argument and will take for granted
that we can integrate by parts for derivatives of $\p$ on $\partial\Omega$.

{\noindent\bf Acknowledgments:}
The authors wish to thank Alessio Figalli for useful comments
and earlier discussions about the subject in the paper.






\section{Preliminaries}
\subsection{Optimal Transport}
We begin with a similar set-up that Gromov has used in establishing the classical isoperimetric inequality.


By the result of Brenier \cite{Brenier} on optimal transportation, given a probability measure $f(x)dx$ on $\Omega$, there exists
 a convex potential function $\p:\mR\rightarrow \mathbb{R}$, such that $\bar{\nabla}\p$ is the unique optimal transport map
from $\Omega$ to $B(0,1)$ (the unit $n$-ball centered at the origin) which pushes forward the probability measure $f(x)dx$, to the probability measure $g(y)dy= \frac{1}{\omega_{n}} dy$ on $B(0,1)$.
We adopt the convention to denote by $\bar{\nabla}\p$, $\bar{\nabla}^2_{ij}\p$ the gradient and the Hessian of $\p$ with respect to the ambient Euclidean metric, in order to distinguish them from $\nabla\p$, $\nabla^2_{\alpha\beta} \p$ (or $\p_{\alpha\beta}$)--the gradient and the Hessian of $\p$ with respect to the metric of $\partial \Omega$. For simplicity, we will denote the boundary $\partial \Omega$ by $M$ from now on.
Since $\bar{\nabla}\p$ preserves the measure, we have the equation
\begin{equation}\label{monge}
\det(\bar{\nabla}^2 \p)(x)=\frac{f(x)}{g(\bar{\nabla} \p(x))}= \omega_{n}f(x).
\end{equation}
The function $f(x)$ will be specified later.
Also $\bar{\nabla}\p$ is the optimal transport map from $\Omega$ to $B(0,1)$. Therefore $|\bar{\nabla}\p|\leq 1$. Thus
\begin{equation}\label{2.59}
|\nabla\p|^2+\p_n^2\leq 1.
\end{equation}
This is a fact that will be frequently used later in the argument.

\noindent Now the convexity of $\p$ implies $\bar{\nabla}^2 \p$ is a positive definite matrix. Therefore by the geometric-arithmetic inequality
$(\det(\bar{\nabla}^2 \p))^{\frac{k+1}{n}}\leq \frac{1}{\binom{n}{k+1}}\sigma_{k+1}(\bar{\nabla}^2 \p)$. Thus we obtain, by integrating over $\Omega$,
that
\begin{equation}\label{2.52}
\begin{split}
\int_{\Omega}(\omega_{n} f(x))^{\frac{k+1}{n}}dx= &\int_{\Omega}(\det(\bar{\nabla}^2 \p))^{\frac{k+1}{n}}dx\\
\leq & \int_{\Omega}\frac{1}{\binom{n}{k+1}} \sigma_{k+1}(\bar{\nabla}^2 \p) dx\\
=&\int_{\Omega}\frac{1}{(k+1)\binom{n}{k+1}}\bar{\nabla}^2_{ij} \p [T_{k}]_{ij}(\bar{\nabla}^2 \p)dx.\\
\end{split}
\end{equation}
Here $[T_{k}]_{ij}$ is the Newton transformation tensor, defined by $[T_{k}]_{ij}(A):= [T_{k}]_{ij}(\overbrace{A,...,A}^k)$ and
\begin{equation}
[T_{k}]_{ij}(A_1,...,A_k):= \frac{1}{k!}\delta_{j,j_1,...,j_k}^{i,i_1,...,i_k} (A_1)_{i_1j_1}\cdots (A_k)_{i_kj_k}.
\end{equation}
For example $$[T_{1}]_{ij}(A)= Tr(A)\delta_{ij}- A_{ij},$$
where $Tr(A)$ is the trace of $A$.
We used in the last line of (\ref{2.52}) that $\sigma_{k+1}(A)= \frac{1}{k+1}A_{ij}  [T_{k}]_{ij}(A) $.
For more properties of $[T_{k}]_{ij}$, we refer the readers to \cite[section 5.1]{ChangWang}.
Using the fact $\partial_j(\bar{\nabla}^2_{ij} \p)= \partial_i (\bar{\nabla}^2_{jj} \p)$, one can easily show that $\partial_{j}([T_{k}]_{ij})(\bar{\nabla}^2 \p)=0$.
Here $\partial_j$ is the coordinate derivative on $\mR$. Hence we have by the divergence theorem that
\begin{equation}\label{2.54}
\begin{split}
&\int_{\Omega} \frac{1}{(k+1)\binom{n}{k+1}}\bar{\nabla}^2_{ij} \p [T_k]_{ij}(\bar{\nabla}^2 \p)dx\\
=&\frac{1}{(k+1)\binom{n}{k+1}}\int_{\Omega} \partial_j\left(\p_i[T_{k}]_{ij}(\bar{\nabla}^2 \p)\right)dx\\
=&\frac{1}{(k+1)\binom{n}{k+1}}\int_{M} [T_{k}]_{ij}(\bar{\nabla}^2 \p )  \p_i n_j d\mu,
\end{split}
\end{equation}
where $n_j$ is the coordinate of the outward unit normal on $M$. If we combine (\ref{2.52}) and (\ref{2.54}) and specify the probability measure
$f(x)dx:=\frac{1}{vol(\Omega)}dx$ on $\Omega$, we then get
\begin{equation}\label{2.54a}
vol(\Omega)^{1-\frac{k+1}{n}}\leq (\frac{1}{\omega_{n}})^{\frac{k+1}{n}}\frac{1}{(k+1)\binom{n}{k+1}} \int_M [T_k]_{ij}(\bar{\nabla}^2\p )\p_i n_jd\mu.
\end{equation}

Comparing the constants in (\ref{2.54a}) (for $k=1$ and 2 respectively) with the constants in (\ref{af1}) and (\ref{af2}) of Theorem
\ref{thm-1'} and \ref{thm-2'} respectively, we
notice that to prove these theorems it suffices to establish inequalities (\ref{2.55}) and (\ref{2.56}) below.
\begin{equation}\label{2.55}
 \int_M [T_1]_{ij}(\bar{\nabla}^2\p )\p_i n_j d\mu\leq \int_M H d\mu,
\end{equation}
\begin{equation}\label{2.56}
 \int_M [T_2]_{ij}(\bar{\nabla}^2\p )\p_i n_j d\mu\leq \int_M \sigma_2(L) d\mu.
\end{equation}

In the following, we will in fact prove two inequalities which are slightly stronger than (\ref{2.55}) and (\ref{2.56}).
\begin{theorem}\label{thm-1}Suppose $L\in \Gamma_2^+$. Then
\begin{equation}\label{af1a}
\int_M [T_1]_{ij}(\bar{\nabla}^2 \p ) \p_i n_jd\mu\leq \int_{M}H-\frac{1}{3} [{T}_1]_{\alpha\beta}(L)\p_\alpha\p_{\beta}d\mu. \end{equation}
\end{theorem}
\begin{theorem}\label{thm-2}Suppose $L\in \Gamma_3^+$. Then
\begin{equation}\label{af2a}
\int_M [T_2]_{ij}(\bar{\nabla}^2 \p ) \p_i n_j d\mu \leq \int_M \sigma_2 (L)-\frac{1}{4}\int_M [{T}_2]_{\alpha\beta}(L)
\p_\alpha\p_{\beta}d\mu.
\end{equation}
\end{theorem}
\noindent Note for $L\in \Gamma_2^+$, $[{T}_1]_{\alpha\beta}(L)\geq 0$; and for $L\in \Gamma_3^+$, $[{T}_2]_{\alpha\beta}(L)\geq 0$.
In section 3 and 4, we will prove these two theorems.

We also remark that by our optimal transport method with approximation argument, $\phi$ is smooth on $M$.
But one only needs to assume $\phi$ is $C^3$, since in the proof we will only take covariant derivatives of $\phi$ up to the third order.

{\it {Proof of Theorem \ref{thm-1'} and \ref{thm-2'}}:}
It is clear that inequality (\ref{af1a}) implies (\ref{2.55}), which in turn implies (\ref{af1}); similarly,
inequality (\ref{af2a}) implies (\ref{2.56}), which in turn implies (\ref{af2}).
Now we will show that equalities in (\ref{af1}) and (\ref{af2}) hold only when $\Omega$ is a ball.
In fact we first claim (\ref{af1}) attains
the equality only if \begin{equation}\label{claim1}\bar{\nabla}^2\p(x)= \lambda(x) Id.\end{equation}
The proof of this claim is given in the next paragraph. The same argument applies when
(\ref{af2}) becomes an equality.

We can approximate $\Omega$ by a sequence of subsets $\{\Omega_l\}_{l=1}^{\infty}$ such that $\bar{\Omega}_l\subset \Omega$ and $\partial \Omega_l\rightarrow \partial \Omega$ in $C^3$ norm.
We then apply Theorem 2.1 on each $\Omega_l$ and derive
\begin{equation}\begin{split}
\omega_n^{2/n}vol(\Omega_l) vol(\Omega)^{-2/n}=&\int_{\Omega_l}
 {\det}^{2/n}(\bar{\nabla}^2 \phi )dx\\
 \leq& \frac{2}{n (n-1)} \int_{\Omega_l}\sigma_2 (\bar{\nabla}^2 \phi)dx\\
 \leq &\frac{1}{n (n-1)} \int_{\partial \Omega_l} H_{\partial \Omega_l} d\mu_l. \\
 \end{split}\end{equation}
Here $H_{\partial \Omega_l}$ denotes the mean curvature of boundary of
$\Omega_l$; and $\mu_l$ denotes the area form of the boundary of $\Omega_l$.
Since $\Omega_l$ is contained in the interior of $\Omega$, $\phi$ on $\Omega_l$ is a $C^3$ function up to the boundary \cite{Caffarelli1}, and thus the arguments to prove Theorem 2.1 on $ \Omega_l$ works. (See also the remark right after the statements of Theorem 2.1 and 2.2.) We then take $l\rightarrow \infty$. By dominating convergence theorem, we derive the same inequalities for $\Omega$. Namely,
\begin{equation}\begin{split}
\omega_n^{2/n} vol(\Omega)^{1-2/n}=&\int_{\Omega}
 {\det}^{2/n}(\bar{\nabla}^2 \phi )dx\\
 \leq& \frac{2}{n (n-1)} \int_{\Omega}\sigma_2 (\bar{\nabla}^2 \phi)dx\\
 \leq &\frac{1}{n (n-1)} \int_{\partial \Omega} H d\mu. \\
 \end{split}\end{equation}
Hence equality is attained only if ${\det}^{2/n}(\bar{\nabla}^2 \phi )= \frac{2}{n (n-1)}\sigma_2 (\bar{\nabla}^2 \phi)$, and thus $\bar{\nabla}^2\p(x)= \lambda(x) Id$. This gives the proof of (\ref{claim1}).



Now, by our choice $f(x)=\frac{1}{vol(\Omega)}$ as well as the equation (\ref{monge}), we see that
$\lambda(x)$ must be a constant $\lambda$ on each connected component of $\Omega$, possibly with different values of $\lambda$ on different components. Thus the unique solution $\bar{\nabla}\p $ to the optimal transport problem is a dilation map $\bar{\nabla}\p(x)= \lambda x$ on each connected component of $\Omega$. Hence each connected component of $\Omega$ is a ball. However, if $\Omega $ is the union of two or more balls, we can compute directly that it does not attain the equality in (\ref{af1a}). Therefore $\Omega$ is a ball.



\subsection{Elementary facts}
By Taylor expansion
\begin{equation}\label{defC_k}(1-s^2)^{1/2}=1-\sum_{k=1}^\infty C_k s^{2k}, \end{equation}
where $C_k= \binom {2k} {k}\frac{1}{2^{2k}(2k-1)}= \frac{(2k-3)!!}{k!2^k}\geq 0$.
In the following, we always take $s=|\nabla \p|$, and $\psi:= \sqrt{1-|\nabla \p|^2}$. By (\ref{2.59}), $0\leq s\leq 1$.
Thus we have \\
{\it{\textbf{Fact (a):}}
\begin{equation}\label{defC_k'}\psi=1-\sum_{k=1}^\infty C_k |\nabla \p|^{2k}. \end{equation}
}

\noindent{\it{\textbf{Fact (b):}}
\begin{equation} \label{Factb}  \sum_{k=1}^\infty 2k C_k |\nabla \p|^{2(k-1)} \psi \equiv 1. \end{equation}
}\\
\begin{proof}If we take derivative on both sides of (\ref{defC_k}), then
\begin{equation}
\frac{-s}{(1-s^2)^{1/2}}= -\sum_{k=1}^\infty 2k C_k s^{2k-1}.
\end{equation}
Let $s= |\nabla \phi|$. It then deduces (\ref{Factb}).
\end{proof}
\noindent Define
$$   F(x):= \sum_{k=1}^\infty \frac{k}{k+1} C_k |\nabla \p(x)|^{2(k-1)}, $$
and $$ G(x):= \frac{1}{2}-\sum_{k=1}^\infty \frac{1}{2(k+1)}C_k |\nabla \p(x)|^{2k}. $$
$F$ and $G$ will appear in the proof of Theorem \ref{thm-2}. If we multiply $s$ on both sides of (\ref{defC_k}) and integrate over $[0,s]$, then we derive\\

\noindent {\it{\textbf{Fact (c):}}
$ 3G |\nabla \p|^2= 1-\psi^3.  $
}\\

\noindent By a simple calculation, we also have \\
\noindent {\it{\textbf{Fact (d):}}
$   2G =\psi + F |\nabla \p |^2.    $
}\\

\noindent  {\it{\textbf{Fact (e):}}
$   F \psi \leq \frac{1}{4}.$
}
\begin{proof} of Fact (e) Since $\frac{1}{2(k+1)}\leq \frac{1}{4}$ for $k\geq 1$,
\begin{equation}
\begin{split}
F\psi = & (\sum_{k=1}^\infty \frac{k}{k+1} C_k|\nabla \p|^{2(k-1)} )\psi \\
\leq & \frac{1}{4}\sum_{k=1}^\infty 2k C_k |\nabla \p |^{2(k-1)}\psi= \frac{1}{4},\\
\end{split}
\end{equation}
where the last equality uses Fact (b).
\end{proof}

\section{Proof of Theorem \ref{thm-1}}
Consider the isometric embedding $i:M\rightarrow \mR$, where $M:= \partial \Omega$.
For $x\in M$, one can write the Hessian of $\phi$ in coordinates of tangential derivatives and normal derivatives of $T_xM$. Let
indices $\alpha$, $\beta$ with $\alpha,\beta= 1,...,n-1$ be the tangential directions, $\vec n$ be the outward unit normal direction on $M$,
 and let $i,j$ with $i,j= 1,...,n$ be the coordinates of $\mR$. Recall that we denote by $\p_{\alpha\beta}$ or $\nabla^2_{\alpha\beta} \p$
the Hessian of $\phi$ with respect to the metric of the surface measure on $M$, and by $\bar{\nabla}^2_{ij}\p$ the Hessian of $\p$ with respect to the ambient (Euclidean) metric.
Let $L_{\alpha\beta}(x)$ be the second fundamental form at $x\in M$.\\
It is well known that
$$\bar{\nabla}^2_{\alpha\beta}\p= \p_{\alpha\beta}+L_{\alpha\beta} \p_n,$$
and $$\bar{\nabla}^2_{\alpha n}\p= \nabla_\alpha (\p_n) -L_{\alpha\beta}\p_\beta.$$
Thus we can decompose
$\bar{\nabla}^2_{\alpha\beta}\p= A+ B$,
where
\makeatletter
\renewcommand*\env@matrix[1][*\c@MaxMatrixCols c]{%
  \hskip -\arraycolsep
  \let\@ifnextchar\new@ifnextchar
  \array{#1}}
\makeatother
\begin{equation}\label{A}
A =
 \begin{pmatrix}[ccc|c]
  \cdots &  \cdots & \cdots & \vdots \\
  \cdots & \p_{\alpha\beta}& \cdots & \nabla_\alpha (\p_n) \\
  \cdots &  \cdots & \cdots & \vdots\\
  \hline
  \cdots  &  \nabla_\alpha (\p_n)  &\cdots &\bar{\nabla}^2_{nn}\p   \\
  \end{pmatrix},
 \end{equation}
and
\begin{equation}\label{B}
B =
 \begin{pmatrix}[ccc|c]
 \cdots &  \cdots & \cdots & \vdots \\
  \cdots & L_{\alpha\beta}\p_n & \cdots &  -L_{\alpha\gamma}\p_\beta  \\
  \cdots &  \cdots & \cdots & \vdots\\
  \hline
  \cdots  &  -L_{\alpha\gamma}\p_\beta  &\cdots&  0   \\
 \end{pmatrix}.
 \end{equation}

\noindent Define
\begin{equation}
L_k:= \int_{M} [T_{k-1}]_{ij}(A+B) \p_i n_j d\mu.
\end{equation}
The inequality in Theorem \ref{thm-1} is equivalent to
\begin{equation}
L_2\leq  \int_{M}H -\frac{1}{3}[{T}_{1}]_{\alpha\beta}(L)\p_\alpha\p_\beta d\mu.
\end{equation}
To prove this, we write $L_2 = L_{2,1}+ L_{2,2}$, where
$$L_{2,1}:= \int_{M} [T_{1}]_{ij}(A) \p_i n_j d\mu,$$
and $$L_{2,2}:= \int_{M} [T_{1}]_{ij}(B) \p_i n_j d\mu.$$
\begin{proposition}\label{2.1}
\begin{equation}\label{2.11}
L_{2,1}= 2\int_{M} \Delta \p \p_n d\mu.\\
\end{equation}
\begin{equation}\label{2.12}
L_{2,2}= \int_{M} H \p_n^2+L_{\alpha\beta} \p_\alpha \p_\beta d\mu.\\
\end{equation}
\end{proposition}
\begin{proof}
Recall $$[T_{1}]_{ij}(A)= Tr(A)\delta_{ij}- A_{ij},$$ where $Tr(A)$ denotes the trace of $A$.
\begin{equation}
\begin{split}
L_{2,1}=& \int_{M} (\Delta \p +\p_{nn})\p_n- (\nabla_\alpha(\p_n)\p_\alpha+ \p_{nn}\p_n) d\mu\\
=&\int_{M}\Delta \p\p_n- \nabla_\alpha(\p_n)\p_\alpha d\mu\\
= &\int_{M}2\Delta \p\p_n d\mu.\\
\end{split}
\end{equation}
\begin{equation}
\begin{split}
L_{2,2}=&\int_{M}(Tr(B)\delta_{in}-B_{in})\p_i d\mu\\
=&\int_{M}(Tr(B)\delta_{\alpha n}-B_{\alpha n})\p_\alpha d\mu+\int_{M}(Tr(B)\delta_{nn}-B_{nn})\p_n d\mu\\
=&\int_{M} H \p_n^2+ L_{\alpha\beta} \p_\alpha \p_\beta d\mu.\\
\end{split}
\end{equation}
\end{proof}

We now define $M_{2,1}$, and $M_{2,2}$ the analogous expressions of $L_{2,1}$, $L_{2,2}$ in (\ref{2.11}) and (\ref{2.12}), with
$\p_n$ replaced  by $\psi$:
\begin{definition}
Define
\begin{equation}
M_{2,1}:= 2\int_{M} \Delta \p \psi d\mu,\\
\end{equation}
and
\begin{equation}
M_{2,2}:= \int_{M} H \psi^2+L_{\alpha\beta} \p_\alpha \p_\beta d\mu.\\
\end{equation}
\end{definition}
\begin{proposition}\label{2.2}
\begin{equation}
\begin{split}
M_{2,1}\leq& \frac{2}{3}\int_{M} (H\delta_{\alpha\beta}-L_{\alpha\beta})\p_\alpha \p_\beta d\mu\\
=& \frac{2}{3}\int_{M} [{T}_1]_{\alpha\beta}(L) \p_\alpha \p_\beta d\mu.\\
\end{split}
\end{equation}
\begin{equation}
M_{2,2}=  \int_{M}H- [{T}_{1}]_{\alpha\beta}(L)\p_\alpha \p_\beta  d\mu.\\
\end{equation}
\end{proposition}
We now assert the inequality (\ref{af1a}) in Theorem \ref{thm-1} follows the inequalities in the above proposition.
To see this, we have
\begin{equation}\label{2.4}
\begin{split}
L_2=&L_{2,1}+ L_{2,2}\\
= & M_{2,1}+ M_{2,2} + \int_{M} 2\Delta \p  (\p_n- \psi) + H (\p_n^2- \psi^2) d\mu\\
\end{split}
\end{equation}
We claim that
\begin{equation}\label{2.5}\int_{M} 2\Delta \p (\p_n- \psi) + H (\p_n^2- \psi^2) d\mu\leq 0.\end{equation}
To see (\ref{2.5}), we re-write $\Delta\p$ as $\Delta\p=\bar{\Delta}|_{T_xM}\p- H\p_n$ and obtain
\begin{equation}
\begin{split}
&\int_{M} 2\Delta\p (\p_n- \psi) + H (\p_n^2- \psi^2) d\mu\\
=&\int_{M} 2\bar{\Delta}|_{T_xM}\p (\p_n- \psi) -2H\p_n (\p_n- \psi)+H (\p_n^2- \psi^2) d\mu\\
=  &\int_{M} 2\bar{\Delta}|_{T_xM}\p (\p_n- \psi)- H (\p_n- \psi)^2 d\mu.\\
\end{split}
\end{equation}
Using (\ref{2.59}), $\p_n\leq \psi= \sqrt{1-|\nabla\p|^2}$; we also have $\bar{\Delta}|_{T_xM}\p \geq 0$ and $H\geq 0$.
Thus (\ref{2.5}) holds. We now deduce from (\ref{2.4}), (\ref{2.5}) and Proposition \ref{2.2} that
\begin{equation}
\begin{split}
L_2\leq & M_{2,1}+ M_{2,2}\\
\leq &\int_{M}H- \frac{1}{3}[{T}_{1}]_{\alpha\beta}(L)\p_\alpha \p_\beta  d\mu.\\
\end{split}
\end{equation}
This finishes the proof of Theorem \ref{thm-1}.

We now begin the proof of Proposition \ref{2.2}. The strategy we will apply here is to expand $\psi$ by the Taylor series and to derive a
recursive inequality for each individual term $\int_M\Delta \p |\nabla\p |^{2k}d\mu$ in the Taylor series. Let us begin with the following lemma.
\begin{lemma}\label{2.3}
Define $$J_k:=\int_{M}\Delta \p |\nabla\p |^{2k}d\mu. $$
Then
\begin{equation}
\dstyle J_k= \frac{2k}{2k+1}\int_{M}[{T}_1]_{\alpha\beta}(\nabla^2 \p)\p_\alpha\p_\beta |\nabla \p|^{2(k-1)}d\mu.\\
\end{equation}
\end{lemma}
\begin{proof}
By integration by parts
\begin{equation}
\begin{split}
J_k:=&\int_{M}\Delta \p |\nabla\p |^{2k}d\mu \\
=&\int_{M} -2k \p_\alpha \p_\beta \p_{\alpha\beta} |\nabla\p |^{2(k-1)}d\mu . \\
\end{split}
\end{equation}
By the definition of $[T_1]$, $\p_{\alpha\beta} = \Delta\p\delta_{\alpha\beta}- [{T}_1]_{\alpha\beta}(\nabla^2 \p)  $. Then
\begin{equation}
\begin{split}
J_k=&\int_{M} -2k \Delta\p |\nabla\p |^{2k}d\mu \\
&+\int_{M} 2k [{T}_1]_{\alpha\beta}(\nabla^2 \p)\p_\alpha \p_\beta |\nabla\p |^{2(k-1)}d\mu \\
= &-2kJ_k+ \int_{M} 2k [{T}_1]_{\alpha\beta}(\nabla^2 \p)\p_\alpha \p_\beta |\nabla\p |^{2(k-1)}d\mu. \\
\end{split}
\end{equation}
Thus
\begin{equation}
J_k= \frac{2k}{2k+1}\int_{M}[{T}_1]_{\alpha\beta}(\nabla^2 \p)\p_\alpha\p_\beta |\nabla \p|^{2(k-1)}d\mu.\\
\end{equation}
This finishes the proof of Lemma \ref{2.3}.
\end{proof}
\noindent{\it {Proof of Proposition \ref{2.2}:}
First, notice that
\begin{equation}
\begin{split}
M_{2,1}:=& \int_{M}2\Delta \p \psi d\mu\\
=& \dstyle 2\int_{M}\Delta \p (1-\sum_{k=1}^{\infty} C_k |\nabla\p |^{2k})d\mu. \\
\end{split}
\end{equation}
Since $\int_{M}\Delta \p d\mu=0$,
$$M_{2,1}=\dstyle -2\sum_{k=1}^{\infty}C_k J_k,$$
where $J_k:=\int_{M}\Delta \p |\nabla\p |^{2k}d\mu $.
Using Lemma \ref{2.3}, we get
$$M_{2,1}=-2 \sum_{k=1}^\infty \frac{2k}{2k+1}C_k\int_{M} [{T}_1]_{\alpha\beta}(\nabla^2 \p)\p_\alpha\p_\beta |\nabla \p|^{2(k-1)}d\mu.$$
Since $\nabla^2_{\alpha\beta} \p =\bar{\nabla}^2_{\alpha\beta} \p- L_{\alpha\beta}\p_n$,
\begin{equation}\label{3.19}
\begin{split}
M_{2,1}=&- \sum_{k=1}^\infty \frac{4k}{2k+1} C_k \int_{M}[{T}_1]_{\alpha\beta}(\bar{\nabla}^2 \p)\p_\alpha\p_\beta |\nabla \p|^{2(k-1)}d\mu\\
&+ \sum_{k=1}^\infty\frac{4k}{2k+1}C_k\int_{M} [{T}_1]_{\alpha\beta}(L)\p_n\p_\alpha\p_\beta |\nabla \p|^{2(k-1)}d\mu.\\
\end{split}
\end{equation}
Since $[{T}_1]_{\alpha\beta}(\bar{\nabla}^2 \p)\geq 0$, the first sum in (\ref{3.19}) is non-positive. Also, in the second sum, we notice $\frac{2}{2k+1}\leq \frac{2}{3}$ for any $k\geq 1$. Therefore
\begin{equation}
\begin{split}
M_{2,1}\leq& \frac{2}{3}\int_{M} [{T}_1]_{\alpha\beta}(L)\p_\alpha\p_\beta(\sum_{k=1}^\infty 2kC_k|\nabla \p|^{2(k-1)})\p_n d\mu.\\
\end{split}
\end{equation}
Now by Fact (b),
$$(\sum_{k=1}^\infty 2kC_k|\nabla \p|^{2(k-1)})\p_n\leq (\sum_{k=1}^\infty 2kC_k|\nabla \p|^{2(k-1)})\psi=1;$$
also we have $ [{T}_1]_{\alpha\beta}(L)\p_\alpha\p_\beta\geq 0$.
Therefore $$\frac{2}{3}\int_{M} [{T}_1]_{\alpha\beta}(L)\p_\alpha\p_\beta(\sum_{k=1}^\infty 2kC_k|\nabla \p|^{2(k-1)})\p_n d\mu
\leq \frac{2}{3}\int_{M} [{T}_1]_{\alpha\beta}(L)\p_\alpha\p_\beta d\mu.
 $$
This completes the proof for $M_{2,1}$. \\
To the term $M_{2,2}$, it is straightforward to see that
\begin{equation}
\begin{split}
M_{2,2}:=& \int_{M} H \psi^2+L_{\alpha\beta} \p_\alpha \p_\beta d\mu\\
= & \int_{M} H - H |\nabla\p|^2+L_{\alpha\beta} \p_\alpha \p_\beta d\mu\\
=& \int_{M}H- [{T}_{1}]_{\alpha\beta}(L)\p_\alpha \p_\beta  d\mu.\\
\end{split}
\end{equation}
This finishes the proof of Proposition \ref{2.2}.

\section{
Proof of Theorem \ref{thm-2}}
\noindent In this section, we will prove
$$L_3:= \int_{M}[T_2]_{ij}(\bar{\nabla}^2\p)  \p_i n_j d\mu\leq
\int_{\partial\Omega} \sigma_2 (L)-\frac{1}{4}\int_{M} [{T}_2]_{\alpha\beta}(L)
\p_\alpha\p_{\beta}d\mu.$$
To prove this, we first decompose $L_3$, using the multi-linearity of $[T_2]_{ij}(\cdot)$, into $L_3=L_{3,1}+ L_{3,2}+L_{3,3} $,
where
$$L_{3,1}:=\int_{M}[T_2]_{ij}(A,A)  \p_i n_j d\mu, $$
$$L_{3,2}:=2\int_{M}[T_2]_{ij}(A,B)  \p_i n_j d\mu, $$
and
$$L_{3,3}:=\int_{M}[T_2]_{ij}(B,B)  \p_i n_j d\mu. $$
$A$ and $B$ are as defined in (\ref{A}) and (\ref{B}).
\begin{proposition}\label{3.1}
\begin{equation}
L_{3,1}= 3\int_{M} \sigma_2(\nabla^2\p)\p_n d\mu - \int_{M} Ric_{\alpha\beta}\p_\alpha\p_\beta \p_n d\mu.\\
\end{equation}
\begin{equation}
L_{3,2}= \frac{3}{2}\int_{M} \Sigma_2(\nabla^2\p, L)\p_n^2 d\mu + \int_{M}[{T}_1]_{\alpha\beta}(\nabla^2\p) L_{\beta\gamma} \p_\alpha\p_\gamma d\mu.\\
\end{equation}
And
\begin{equation}
L_{3,3}= \int_{M} \sigma_{2}(L)\p_n^3 d\mu +\int_{M} [{T}_1]_{\alpha\beta}(L)L_{\beta\gamma}\p_\alpha\p_\gamma \p_n d\mu .\\
\end{equation}
\end{proposition}
Here $\Sigma_2(A, B):= Tr(A)Tr(B)-Tr(AB)$ for any two tensors $A$ and $B$. It is the linear polarization of $\sigma_2(\cdot)$ (up to a multiplicative constant) in the sense that $\Sigma_2(A, A)= 2\sigma_2(A)$, and it is symmetric and linear with respect to $A$ and $B$. For polarization of $\sigma_k$, we refer to \cite{ChangWang, ChangWang1}.
The proof of Proposition \ref{3.1} is by direct computation. Thus we omit it here. Now as what we did in the previous section, we define $M_{3,1}$, $M_{3,2}$, $M_{3,3}$, simply by substituting $\p_n$ by $\psi$ in the formulas of Proposition \ref{3.1}.
\begin{definition}\label{4.3}
\begin{equation}
M_{3,1}= 3\int_{M} \sigma_2(\nabla^2\p)\psi d\mu - \int_{M} Ric_{\alpha\beta}\p_\alpha\p_\beta \psi d\mu.\\
\end{equation}
\begin{equation}
M_{3,2}= \frac{3}{2}\int_{M} \Sigma_2(\nabla^2\p, L)\psi^2 d\mu + \int_{M}[{T}_1]_{\alpha\beta}(\nabla^2\p) L_{\beta\gamma} \p_\alpha\p_\gamma d\mu.\\
\end{equation}
And
\begin{equation}
M_{3,3}= \int_{M} \sigma_{2}(L)\psi^3 d\mu +\int_{M} [{T}_1]_{\alpha\beta}(L)L_{\beta\gamma}\p_\alpha\p_\gamma \psi d\mu .\\
\end{equation}
\end{definition}
We first simplify the formula of $M_{3,2}$.
\begin{proposition}\label{3.2}
\begin{equation}
M_{3,2}=-2\int_{M} [{T}_2]_{\alpha\beta}(\nabla^2\p ,L) \p_\alpha\p_\beta d\mu.
\end{equation}
\end{proposition}

\begin{proof}
Since $\psi^2 = 1-|\nabla \p|^2$, applying integration by parts we have
\begin{equation}
\begin{split}
&\int_{M} \Sigma_2(\nabla^2\p, L)\psi^2 d\mu\\
= &\int_{M} (\Delta \p H-\p_{\alpha\beta} L_{\alpha\beta})(1-|\nabla \p|^2) d\mu\\
=&-\int_{M}  \p_\alpha H_\alpha (1-|\nabla \p|^2) + \p_\alpha H (-2\p_\gamma \p_{\gamma\alpha}) d\mu\\
&+ \int_{M}  \p_\alpha L_{\alpha\beta,\beta} (1-|\nabla \p|^2) + \p_\alpha L_{\alpha\beta}  (-2\p_\gamma \p_{\gamma\beta}) d\mu\\
\end{split}
\end{equation}
By Codazzi equation $H_\alpha=L_{\alpha\beta,\beta} $. Hence the first term and the third term in the last equality above are canceled. So we have
\begin{equation}\label{formula3.2}
\begin{split}
&\int_{M} \Sigma_2(\nabla^2\p, L)\psi^2 d\mu\\
=& 2\int_{M} (\p_\alpha H\p_\gamma \p_{\gamma\alpha}-  \p_\alpha L_{\alpha\beta}\p_\gamma \p_{\gamma\beta} )d\mu\\
=& 2\int_{M}  \p_\alpha \p_\gamma [{T}_1]_{\alpha\beta}(L)\p_{\beta\gamma}d\mu.\\
\end{split}
\end{equation}
Substituting $\frac{1}{2}\int_{M} \Sigma_2(\nabla^2\p, L)\psi^2 d\mu$ by formula (\ref{formula3.2}) in $M_{3,2}$, we have
\begin{equation}\label{3.18}
\begin{split}
M_{3,2}=&(1+\frac{1}{2})\int_{M} \Sigma_2(\nabla^2\p, L)\psi^2 d\mu +\int_{M}[{T}_1]_{\alpha\beta}(\nabla^2\p) L_{\beta\gamma} \p_\alpha\p_\gamma d\mu\\
=&\int_{M} \Sigma_{2}(\nabla^2 \p, L) \psi^2 + \p_\alpha\p_\gamma [{T}_1]_{\alpha\beta}(L)\p_{\beta\gamma}+ [{T}_1]_{\alpha\beta}(\nabla^2 \p)L_{\beta\gamma} \p_\alpha\p_\gamma d\mu\\
=&\int_{M} \Sigma_{2}(\nabla^2 \p, L) (1-|\nabla \p|^2) + \p_\alpha\p_\gamma [{T}_1]_{\alpha\beta}(L)\p_{\beta\gamma}\\
&+ [{T}_1]_{\alpha\beta}(\nabla^2 \p)L_{\beta\gamma} \p_\alpha\p_\gamma d\mu.\\
\end{split}
\end{equation}
By the Codazzi equation again,
\begin{equation}
\begin{split}
\int_{M}\Sigma_{2}(\nabla^2 \p, L)d\mu=& \int_{M}\p_\alpha H_\alpha- \p_\alpha L_{\alpha\beta,\beta} d\mu\\
=0,\\
\end{split}
\end{equation}
Therefore (\ref{3.18}) implies
\begin{equation}
\begin{split}
M_{3,2}=&\int_{M} \Sigma_{2}(\nabla^2 \p, L) (-|\nabla\p|^2) + \p_\alpha\p_\gamma [{T}_1]_{\alpha\beta}(L)\p_{\beta\gamma}\\
&+ [T_1]_{\alpha\beta}(\nabla^2 \p)L_{\beta\gamma} \p_\alpha\p_\gamma d\mu.\\
\end{split}
\end{equation}
Using the fact that for tensors $A_{\alpha\beta}$, $B_{\alpha\beta}$
\begin{equation}
\begin{split}
2[T_2]_{\alpha\gamma}(A, B)= \Sigma_{2}(A, B)\delta_{\alpha\gamma} -  [T_1]_{\alpha\beta}(A)B_{\beta\gamma}-  [T_1]_{\alpha\beta}(B)A_{\beta\gamma},
\end{split}
\end{equation}
we have
\begin{equation}
M_{3,2}= -2 \int_{M} [{T}_2]_{\alpha\beta}(\nabla^2\p ,L) \p_\alpha\p_\beta d\mu.
\end{equation}

\end{proof}
\begin{proposition}\label{3.4}Define $$A_k:=\int_{M} \sigma_2(\nabla^2 \p)|\nabla \p|^{2k}d\mu.$$
Then
\begin{equation}
\begin{split}
A_k=&\frac{k}{k+1}\int_{M} [{T}_2]_{\alpha\beta}(\nabla^2\p,\nabla^2\p  ) \p_\alpha\p_\beta |\nabla \p|^{2(k-1)} d\mu\\
&+\frac{1}{2(k+1)} \int_{M} Ric_{\alpha\beta}\p_\alpha\p_\beta|\nabla \p|^{2k} d\mu.\\
\end{split}
\end{equation}
\end{proposition}

\begin{proof}Applying properties of $T_1$ and the integration by parts, we have
\begin{equation}\label{4.1}
\begin{split}
A_k:=&\int_{M} \sigma_2(\nabla^2 \p)|\nabla \p|^{2k}d\mu\\
=&\int_{M} \frac{1}{2}\p_{\alpha\beta} [{T}_1]_{\alpha\beta}(\nabla^2 \p)|\nabla\p|^{2k}d\mu\\
=& \int_{M}-\frac{1}{2} \p_\alpha ([{T}_1]_{\alpha\beta}(\nabla^2 \p))_\beta|\nabla \p|^{2k}-\frac{1}{2}\p_\alpha [{T}_1]_{\alpha\beta}(\nabla^2 \p)(|\nabla \p|^{2k})_\beta d\mu.\\
\end{split}
\end{equation}
\begin{equation}
\begin{split}
([T_1]_{\alpha\beta}(\nabla^2\p))_\beta= (\Delta \p \delta_{\alpha\beta}-\p_{\alpha\beta})_\beta
= -Ric_{\alpha\beta}\p_\beta,
\end{split}
\end{equation}
where $Ric_{\alpha\beta}$ denotes the Ricci curvature of $M$. Therefore (\ref{4.1}) becomes
\begin{equation}
\begin{split}
A_k=&\int_{M} \frac{1}{2} \p_\alpha\p_\beta Ric_{\alpha\beta}|\nabla \p|^{2k}- k\p_\alpha [{T}_1]_{\alpha\beta}(\nabla^2\p) \p_\gamma \p_{\gamma\beta} |\nabla \p|^{2(k-1)}  d\mu.\\
\end{split}
\end{equation}
We now notice that $[{T}_1]_{\alpha\beta}(\nabla^2\p)\p_{\gamma\beta}= \sigma_2(\nabla^2 \p)\delta_{\alpha\gamma}- [{T}_2]_{\alpha\gamma}(\nabla^2 \p,\nabla^2\p )$.
Thus
\begin{equation}
\begin{split}
A_k=&\int_{M} [\frac{1}{2} \p_\alpha\p_\beta Ric_{\alpha\beta}|\nabla \p|^{2k}\\
&- k\p_\alpha \left(\sigma_2(\nabla^2 \p)\delta_{\alpha\gamma}- [{T}_2]_{\alpha\gamma}(\nabla^2\p,\nabla^2\p )\right)\p_\gamma |\nabla \p|^{2(k-1)} ] d\mu\\
=&\int_{M} \frac{1}{2} \p_\alpha\p_\beta Ric_{\alpha\beta}|\nabla \p|^{2k}- kA_k+ k\p_\alpha [{T}_2]_{\alpha\gamma}(\nabla^2\p,\nabla^2\p )\p_\gamma |\nabla \p|^{2(k-1)}  d\mu.\\
\end{split}
\end{equation}
Therefore
\begin{equation}
\begin{split}
A_k=&\frac{k}{k+1}\int_{M} [{T}_2]_{\alpha\beta}(\nabla^2\p ,\nabla^2\p ) \p_\alpha\p_\beta |\nabla \p|^{2(k-1)} d\mu\\
&+\frac{1}{2(k+1)} \int_{M} Ric_{\alpha\beta}\p_\alpha\p_\beta|\nabla \p|^{2k} d\mu.\\
\end{split}
\end{equation}

\end{proof}
\begin{corollary}\label{3.5}
\begin{equation}\label{3.51}
\begin{split}
M_{3,1}= & 3\int_{M} \sigma_2(\nabla^2 \p)\psi d\mu- \int_{M} Ric_{\alpha\beta}\p_\alpha\p_\beta \psi d\mu\\
=& -3 \int_{M}[{T}_2]_{\alpha\beta}(\nabla^2\p ,\nabla^2\p) \p_\alpha\p_\beta F d\mu\\
&+3\int_{M} Ric_{\alpha\beta}\p_\alpha\p_\beta G  d\mu - \int_{M} Ric_{\alpha\beta}\p_\alpha\p_\beta \psi d\mu,\\
\end{split}
\end{equation}
where $$F(x):= \sum_{k=1}^\infty \frac{k}{k+1}C_k |\nabla \p|^{2(k-1)}. $$
$$G(x):= \frac{1}{2}- \sum_{k=1}^\infty \frac{1}{2(k+1)}C_k |\nabla \p|^{2k}.$$
\end{corollary}
\begin{proof}
By Fact (a),
\begin{equation}\label{3.50}\begin{split}
M_{3,1}=&3\int_{M} \sigma_2(\nabla^2 \p)\psi d\mu- \int_{M} Ric_{\alpha\beta}\p_\alpha\p_\beta \psi d\mu\\
=&3\int_{M} \sigma_2(\nabla^2 \p)(1-\sum_{k=1}^\infty C_k |\nabla \p|^{2k}) d\mu- \int_{M} Ric_{\alpha\beta}\p_\alpha\p_\beta \psi d\mu\\
=& 3\int_{M} \sigma_2(\nabla^2 \p)d\mu -3 \sum_{k=1}^\infty C_k A_k - \int_{M} Ric_{\alpha\beta}\p_\alpha\p_\beta \psi d\mu.\\
\end{split}
\end{equation}
Substituting the formula of $A_k$ in Proposition \ref{3.4} and the equality that
\begin{equation}\begin{split}\int_{M} \sigma_2(\nabla^2 \p)d\mu=&\frac{1}{2}\int_{M} \p_{\alpha\beta}(\Delta \p \delta_{\alpha\beta}- \p_{\alpha\beta})  d\mu\\
=&\frac{1}{2}\int_{M} -\p_\alpha \left( (\Delta \p)_\alpha- \p_{\alpha\beta,\beta}\right)  d\mu\\
=& \frac{1}{2}\int_{M} Ric_{\alpha\beta}\p_\alpha\p_\beta  d\mu, \end{split}\end{equation}
in (\ref{3.50}), we derive (\ref{3.51}).
\end{proof}
\noindent Applying Proposition \ref{3.2} and Corollary \ref{3.5}, we get
\begin{equation}\label{star3''}
\begin{split}
 &M_{3,1}+M_{3,2}+M_{3,3}\\
=&-3 \int_{M}[{T}_2]_{\alpha\beta}(\nabla^2\p ,\nabla^2\p) \p_\alpha\p_\beta F d\mu+
3\int_{M} Ric_{\alpha\beta}\p_\alpha\p_\beta G  d\mu\\
&-2  \int_{M}[{T}_2]_{\alpha\beta}(\nabla^2\p ,L)\p_\alpha\p_\beta d\mu+\int_{M} \sigma_2(L)\psi^3 d\mu\\
&+\int_{M} Ric_{\alpha\beta}\p_\alpha\p_\beta \psi d\mu-\int_{M} [T_1]_{\alpha\beta}(L)L_{\beta\gamma}\p_\alpha\p_\gamma \psi d\mu.
\end{split}
\end{equation}
By the Gauss equation, \begin{equation}\label{Gauss}Ric_{\alpha\beta} = [T_1]_{\alpha\gamma}(L)L_{\gamma\beta}.\end{equation}
Thus the last two terms in (\ref{star3''}) are canceled.
Therefore
\begin{equation}\label{star3}
\begin{split}
 &M_{3,1}+M_{3,2}+M_{3,3}\\
=&-3 \int_{M}[{T}_2]_{\alpha\beta}(\nabla^2\p ,\nabla^2\p) \p_\alpha\p_\beta F d\mu+
3\int_{M} Ric_{\alpha\beta}\p_\alpha\p_\beta G  d\mu\\
&-2  \int_{M}[{T}_2]_{\alpha\beta}(\nabla^2\p ,L)\p_\alpha\p_\beta d\mu+\int_{M} \sigma_2(L)\psi^3 d\mu.\\
\end{split}
\end{equation}
We now apply Fact (c): $3G|\nabla \p|^2= 1-\psi^3$ to combine the second and the last term in (\ref{star3})
\begin{equation}\label{4.2}
\begin{split}
&3\int_{M} Ric_{\alpha\beta}\p_\alpha\p_\beta G  d\mu+ \int_{M} \sigma_2(L)\psi^3 d\mu\\
=& 3\int_{M} Ric_{\alpha\beta}\p_\alpha\p_\beta G  d\mu +\int_{M} \sigma_2(L)(1-3G|\nabla \p|^2) d\mu\\
\end{split}
\end{equation}
By (\ref{Gauss}) and the fact that
\begin{equation}
[T_1]_{\alpha\gamma}(L)L_{\gamma\beta}- \sigma_2(L)\delta_{\alpha\beta}
=[{T}_2]_{\alpha\beta}(L,L),
\end{equation}
(\ref{4.2}) is equal to
\begin{equation}
\begin{split}
&3\int_{M} [T_1]_{\alpha\gamma}(L)L_{\gamma\beta}\p_\alpha\p_\beta G  d\mu +\int_{M} \sigma_2(L)(1-3G|\nabla \p|^2) d\mu\\
=&\int_{M} \sigma_2(L) d\mu- 3\int_{M}[{T}_2]_{\alpha\beta}(L,L)\p_\alpha\p_\beta G d\mu.
\end{split}
\end{equation}
In conclusion (\ref{star3}) is deduced to
\begin{equation}\label{star3'}
M_{3,1}+M_{3,2}+M_{3,3}=\int_{M} \sigma_2(L) d\mu+ E_1+ E_2 + E_3,
\end{equation}
where
\begin{equation}
\begin{split}
E_1=& -3 \int_{M}[{T}_2]_{\alpha\beta}(\nabla^2\p ,\nabla^2\p) \p_\alpha\p_\beta F d\mu\\
E_2=& -2  \int_{M}[{T}_2]_{\alpha\beta}(\nabla^2\p ,L)\p_\alpha\p_\beta d\mu\\
E_3= &- 3\int_{M}[{T}_2]_{\alpha\beta}(L,L)\p_\alpha\p_\beta G d\mu.\\
\end{split}
\end{equation}
In the following we will prove
\begin{proposition}\label{3.12}
\begin{equation}E_1+ E_2+ E_3\leq -\frac{1}{4}\int_{M}[{T}_2]_{\alpha\beta}(L)\p_\alpha\p_\beta d\mu. \end{equation}
From this, it is obvious that
\begin{equation}\begin{split}
M_{3,1}+M_{3,2}+M_{3,3}\leq& \int_{M} \sigma_2(L) d\mu-\frac{1}{4}\int_{M}[{T}_2]_{\alpha\beta}(L)\p_\alpha\p_\beta d\mu.\\
\end{split}\end{equation}
\end{proposition}
\begin{proof} By $\nabla^2_{\alpha\beta} \p= \bar{\nabla}^2_{\alpha\beta}\p- L_{\alpha\beta}\p_n$, we obtain
\begin{equation}
\begin{split}
E_1=& -3 \int_{M} [{T}_2]_{\alpha\beta}(\bar{\nabla}^2\p- L\p_n, \bar{\nabla}^2\p- L\p_n)\p_\alpha\p_\beta F d\mu\\
=& -3 \int_{M} [{T}_2]_{\alpha\beta}(\bar{\nabla}^2\p, \bar{\nabla}^2\p)\p_\alpha\p_\beta F d\mu\\
&+6 \int_{M} [{T}_2]_{\alpha\beta}(\bar{\nabla}^2\p, L)\p_n\p_\alpha\p_\beta F d\mu\\
&-3 \int_{M} [{T}_2]_{\alpha\beta}(L, L)\p_n^2 \p_\alpha\p_\beta F d\mu\\
=:& E_{1,1}+ E_{1,2}+ E_{1,3}\\
\end{split}
\end{equation}
Since $\p$ is a convex function, $\bar{\nabla}^2\p$ is positive definite. Therefore
\begin{equation}
[{T}_2]_{\alpha\beta} (\bar{\nabla}^2\p, \bar{\nabla}^2\p)\geq 0.
\end{equation}
This together with $F\geq 0$ implies that $E_{1,1}\leq 0 $. So $E_1\leq E_{1,2}+ E_{1,3}.$

For $E_2$, using $\nabla^2_{\alpha\beta} \p= \bar{\nabla}^2_{\alpha\beta}\p- L_{\alpha\beta}\p_n$, we get
\begin{equation}
\begin{split}
E_2=& -2 \int_{M} [{T}_2]_{\alpha\beta}(\bar{\nabla}^2\p , L)\p_\alpha\p_\beta d\mu\\
&+ 2 \int_{M} [{T}_2]_{\alpha\beta}(L, L)\p_n \p_\alpha\p_\beta d\mu\\
:=& E_{2,1}+ E_{2,2}.\\
\end{split}
\end{equation}
We next observe that
\begin{lemma}\label{3.6}
\begin{equation}
E_{1,2}+ E_{2,1}\leq -\frac{1}{2}\int_{M} [{T}_2]_{\alpha\beta}(\bar{\nabla}^2\p , L)\p_\alpha\p_\beta d\mu\leq 0.
\end{equation}
\end{lemma}
\begin{proof}
\begin{equation}
E_{1,2}+ E_{2,1}= \int_{M} [{T}_2]_{\alpha\beta}(\bar{\nabla}^2\p , L)\p_\alpha\p_\beta(6F\p_n-2) d\mu.
\end{equation}
On the one hand, $6F\p_n-2\leq -\frac{1}{2}$ because $F\p_n\leq F\psi$, and $F\psi \leq \frac{1}{4}$ (Fact (e)); on the other hand, since $L_{\alpha\beta}\in \Gamma_3^+$ and $\bar{\nabla}^2\p\geq 0$ (positive definite), we have
\begin{equation}
[{T}_2]_{\alpha\beta} (\bar{\nabla}^2\p, L)\geq 0.
\end{equation}
Therefore \begin{equation}E_{1,2}+ E_{2,1}\leq -\frac{1}{2}\int_{M} [{T}_2]_{\alpha\beta}(\bar{\nabla}^2\p , L)\p_\alpha\p_\beta d\mu\leq 0.\end{equation}
\end{proof}
\noindent We continue the proof of Proposition \ref{3.12}. Applying Lemma \ref{3.6}, we have
\begin{equation}\begin{split}\label{3.10}
E_1+ E_2+ E_3\leq &E_{1,3}+ E_{2,2}+ E_3\\
= &\int_{M} [{T}_2]_{\alpha\beta}(L)\p_\alpha\p_\beta P(x) d\mu,
\end{split}
\end{equation}
where $P(x):= -3 F\p_n^2+ 2\p_n-3G $.
By Fact (d), $2G= \psi+ F|\nabla \p|^2$.
Thus $$P=-3 F\p_n^2+ 2\p_n-\frac{3}{2}\psi -\frac{3}{2}|\nabla \p|^2.$$
It is not hard to show $P\leq -\frac{1}{4}$.
In fact,
\begin{equation}\label{3.7}
\begin{split}
P=&-3 F\p_n^2+ 2\p_n-\frac{3}{2}\psi -\frac{3}{2}F|\nabla \p|^2\\
= & -3 \sum_{k=1}^{\infty}\frac{k}{k+1}C_k |\nabla \p|^{2(k-1)}\p_n^2+ 2\p_n-\frac{3}{2}\psi -\frac{3}{2} \sum_{k=1}^{\infty}\frac{k}{k+1}C_k |\nabla \p|^{2k}.\\
\end{split}
\end{equation}
To estimate $-3 \sum_{k=1}^{\infty}\frac{k}{k+1}C_k |\nabla \p|^{2(k-1)}\p_n^2$ in (\ref{3.7}), we observe $C_k\geq 0$, and $C_1=\frac{1}{2}$. So
we drop the sum over $k\geq 2$ and get
\begin{equation}\label{3.8}-3\sum_{k=1}^{\infty}\frac{k}{k+1}C_k |\nabla \p|^{2(k-1)}\p_n^2
\leq - \frac{3}{4} \p_n^2.
\end{equation}
And to estimate $-\frac{3}{2} \sum_{k=1}^{\infty}\frac{k}{k+1}C_k |\nabla \p|^{2k}$ in (\ref{3.7}), we use
the fact that $\frac{k}{k+1}\geq\frac{1}{2}$ for all $k\geq1$, and Fact (a): $\psi=1- \sum_{k=1}^{\infty} C_k |\nabla \p|^{2k}$
to get
\begin{equation}\label{3.9}-\frac{3}{2} \sum_{k=1}^{\infty}\frac{k}{k+1}C_k |\nabla \p|^{2k}
\leq -\frac{3}{4}\sum_{k=1}^{\infty} C_k |\nabla \p|^{2k}=  -\frac{3}{4} (1-\psi).
\end{equation}
Applying (\ref{3.8}) and (\ref{3.9}) to (\ref{3.7}), we have
\begin{equation}
\begin{split}
P\leq&  - \frac{3}{4} \p_n^2   + 2\p_n-\frac{3}{2}\psi-\frac{3}{4} (1-\psi)\\
\leq & - \frac{3}{4} \p_n^2   + 2\p_n-\frac{3}{4}\p_n-\frac{3}{4},\\
\end{split}\end{equation}
due to the fact that $\p_n\leq \psi$. Now
\begin{equation}
\begin{split}
 - \frac{3}{4} \p_n^2   + 2\p_n-\frac{3}{4}\p_n-\frac{3}{4}
= -\frac{1}{4}(3\p_n-2)(\p_n-1)-\frac{1}{4}\leq -\frac{1}{4},
\end{split}
\end{equation}
for $\p_n\leq 1$. Thus we obtain $P\leq -\frac{1}{4}$. By (\ref{3.10}), it concludes that
\begin{equation}E_1+ E_2+ E_3\leq -\frac{1}{4}\int_{M} [{T}_2]_{\alpha\beta}(L)\p_\alpha\p_\beta d\mu\leq0.
\end{equation}
This completes the proof of Proposition \ref{3.12}
\end{proof}
Finally, we are ready to give the proof of Theorem \ref{thm-2} using Proposition \ref{3.12}.
\begin{proof}By Proposition \ref{3.1} and Definition \ref{4.3}
\begin{equation}\label{3.14}
\begin{split}
L_3=& L_{3,1}+ L_{3,2}+ L_{3,3}\\
= &  M_{3,1}+ M_{3,2}+ M_{3,3}+ \int_{M}3\sigma_2(\nabla^2 \p)(\p_n-\psi)\\
&+ \frac{3}{2}\Sigma_{2}(\nabla^2 \p, L)
(\p_n^2-\psi^2)+ \sigma_2(L)(\p_n^3-\psi^3)d\mu.\\
\end{split}
\end{equation}
We will prove
\begin{equation}\label{4.4} \int_{M}3\sigma_2(\nabla^2 \p)(\p_n-\psi)+\frac{3}{2}\Sigma_{2}(\nabla^2 \p, L)
(\p_n^2-\psi^2)+ \sigma_2(L)(\p_n^3-\psi^3)d\mu\leq0 .\end{equation}
Recall that $\Sigma_2(A,B):= Tr(A)Tr(B)- Tr(AB)$.
This immediately implies that $$\sigma_2 (A)= \frac{1}{2}\Sigma_{2}(A,A).$$
Using $\nabla^2_{\alpha\beta} \p= \bar{\nabla}^2_{\alpha\beta}\p- L_{\alpha\beta}\p_n$ and
the linearity of $\Sigma_2(\cdot, \cdot)$, first of all
\begin{equation}\label{3.15'}
\begin{split}
&3\int_{M}\sigma_2(\nabla^2 \p)(\p_n-\psi)d\mu\\
=& \frac{3}{2}\int_{M}\Sigma_2(\bar{\nabla}^2_{\alpha\beta}\p- L_{\alpha\beta}\p_n,\bar{\nabla}^2_{\alpha\beta}\p- L_{\alpha\beta}\p_n)(\p_n-\psi)d\mu\\
= &    \int_{M}3 \sigma_2(\bar{\nabla}^2 \p)(\p_n-\psi) -3\Sigma_2(\bar{\nabla}^2 \p, L)\p_n(\p_n-\psi)+
3\sigma_2(L)\p_n^2 (\p_n-\psi)d\mu.\\
\end{split}
\end{equation}
Since $\sigma_2(\bar{\nabla}^2 \p)\geq 0$ and $\p_n\leq \psi$, the first term in the above line is non-positive.
Thus
\begin{equation}\label{3.15}
\begin{split}
&3\int_{M}\sigma_2(\nabla^2 \p)(\p_n-\psi)d\mu\\
\leq& \int_{M}-3\Sigma_2(\bar{\nabla}^2 \p, L)\p_n(\p_n-\psi)+
3\sigma_2(L)\p_n^2 (\p_n-\psi)d\mu,\\
\end{split}
\end{equation}
Secondly,
\begin{equation}\label{3.16}
\begin{split}
 &\frac{3}{2}\Sigma_{2}(\nabla^2 \p, L) (\p_n^2-\psi^2)\\
=&\frac{3}{2}\Sigma_{2}(\bar{\nabla}^2 \p, L) (\p_n^2-\psi^2)-3 \sigma_2(L)\p_n(\p_n^2-\psi^2).\\
\end{split}
\end{equation}
Using (\ref{3.15}) and (\ref{3.16}), we have
\begin{equation}\label{3.17}
\begin{split}
& \int_{M}3\sigma_2(\nabla^2 \p)(\p_n-\psi)+\frac{3}{2}\Sigma_{2}(\nabla^2 \p, L)
(\p_n^2-\psi^2)+ \sigma_2(L)(\p_n^3-\psi^3)d\mu\\
\leq&-3 \int_{M}\Sigma_{2}(\bar{\nabla}^2 \p, L)\left(\p_n(\p_n-\psi)-\frac{1}{2}(\p_n^2-\psi^2)\right) d\mu\\
&+ \int_{M}\sigma_{2}(L)
\left(3\p_n^2 (\p_n-\psi)-3\p_n(\p_n^2- \psi^2)+ (\p_n^3-\psi^3)\right) d\mu\\
=&-\frac{3}{2} \int_{M}\Sigma_{2}(\bar{\nabla}^2 \p, L)(\p_n-\psi)^2d\mu\\
&+\int_{M}\sigma_{2}(L)(\p_n-\psi)^3  d\mu.\\
\end{split}
\end{equation}
This is less or equal than 0, due to the fact that $\Sigma_{2}(\bar{\nabla}^2 \p, L)\geq 0$, $\sigma_{2}(L)\geq 0$ and $\p_n\leq \psi$.
Thus (\ref{4.4}) is proved.
Plugging (\ref{4.4}) into (\ref{3.14}) and applying Proposition \ref{3.12}, we conclude that
\begin{equation}\label{f3.18}
\begin{split}
L_3\leq& M_{3,1}+ M_{3,2}+ M_{3,3}\\
\leq &\int_{M}\sigma_2(L)d\mu-\frac{1}{4}\int_{M}[{T}_2]_{\alpha\beta}(L)\p_\alpha\p_\beta d\mu.\\
\end{split}\end{equation}
This completes the proof of Theorem \ref{thm-2}.
\end{proof}

\end{document}